\documentclass[12pt]{article}
\usepackage{amsmath,amssymb,enumerate,amsthm}
\usepackage[colorlinks,bookmarksopen,bookmarksnumbered,citecolor=red]{hyperref}
\usepackage{graphicx}
\usepackage{amsmath}
\numberwithin{equation}{section}
\usepackage{amsfonts}
\usepackage{a4wide}
\usepackage{mathrsfs}
\usepackage[numbers,sort&compress]{natbib}
\newtheorem{thm}{Theorem}[section]
\newtheorem{lem}[thm]{Lemma}
\newtheorem{prop}[thm]{Proposition}
\newtheorem{defn}[thm]{Definition}
\newtheorem{rem}[thm]{Remark}

\title{Sharp Exponent of Stable Standing Waves  for the Perturbated Hartree Equation}
\author{Guoyi Fu, Shanshan Fu, Xiaoguang Li, Jian  Zhang, Shihui Zhu\thanks{Corresponding author \textit{E-mail address}: shihuizhumath@163.com}}
\date{}
\begin{document}
	\maketitle
	\begin{abstract}
		This paper is concerned with the stability of standing waves for the mass-critical Hartree equation with a focusing perturbation by the variational method.  The profile decomposition theory is employed to prove the attainability of the cross constrained variational problem, and then the comparison of two cross constrained  variational problems is derived. The sharp criteria of blowup,  the orbital stability, and strong instability of standing waves without any frequency constraint are obtained. This improves the  cross constrained variational argument proposed by Zhang (2005).
	\end{abstract}
	
\section{Introduction}
\subsection{Setting of the problem}
We consider the following mass-critical Hartree equation with a  focusing perturbation 
\begin{equation}\label{Hartree equation}
	i\psi_{t}+\Delta\psi+(|x|^{-2}*|\psi|^{2})\psi+|\psi|^{p-1}\psi=0, \qquad (t, x)\in\mathbb{R}^{+}\times\mathbb{R}^{D}, 
\end{equation}
where $\psi:=\psi(t, x)$ is a complex valued wave function, $i^{2}=-1$, $D\ge 3$, $1<p<1+\frac{4}{D-2}$, $\Delta$ is the Laplace operator, and $*$ stands for the convolution. 
Impose the initial value of Eq. \eqref{Hartree equation},
	\begin{equation}\label{initial data}
		\psi(0, x)=\psi_{0}, \qquad x\in\mathbb{R}^{D}. 
	\end{equation}
Fang and Han \cite{FH2011} established the local well-posedness in $H^{1}(\mathbb{R}^{D})$, assuming that  $D\ge 3$ and $1<p<1+\frac{4}{D-2}$, if the initial datum $\psi_{0}\in H^{1}(\mathbb{R}^{D})$, then there exists a unique solution $\psi(t, x)$ of the Cauchy problem (\ref{Hartree equation})-(\ref{initial data}) such that $\psi(t, x)\in C([0, T); H^{1}(\mathbb{R}^{D}))$ and either $T=+\infty$ (global existence) or else $T<+\infty$ and 
	\begin{equation*}
		\lim_{t\to T}\|\psi(t)\|_{H^{1}(\mathbb{R}^{D})}=+\infty\quad\text{(blow-up)}. 
	\end{equation*}
	Besides, the following conservation laws holds for all $t\in[0,T)$. 
		 \begin{itemize}
		\item   Conservation of mass
		\begin{equation}
			\mathscr{M}[\psi]=\int_{\mathbb{R}^{D}} |\psi|^{2}dx=\mathscr{M}[\psi_{0}].
		\end{equation}
		\item  Conservation of energy
		\begin{equation}
			\mathscr{E}[\psi]=\frac{1}{2}\int_{\mathbb{R}^{D}} |\nabla\psi|^{2}dx-\frac{1}{4}\int_{\mathbb{R}^{D}}(|x|^{-2}*|\psi|^{2})|\psi|^{2}dx-\frac{1}{p+1}\int_{\mathbb{R}^{D}} |\psi|^{p+1}dx=\mathscr{E}[\psi_0]. 
		\end{equation} 
	\end{itemize}


In this present paper, we investigate sharp criteria  of blowup and stability of standing waves
for Eq. \eqref{Hartree equation}. Indeed, in the first part, with the spirit of cross constrained variational method, we investigate the new variational characterization by profile decomposition, see Proposition~\ref{noempty}. Then, the comparison of two constrained variational problems is obtained, see Proposition \ref{compare}, which may answer the open problem arose in \cite{Zhang2005}. On the basis of these results, we establish the sharp threshold for global existence and blow up of solutions, see Theorem \ref{Theorem 4.1}. 

For the rest of this paper, we are focusing on the sharp criteria for stability of standing waves for Eq. \eqref{Hartree equation}. More precisely, when $1<p<1+\frac{4}{D}$, by utilizing profile decomposition, we prove the orbital stability of standing waves under condition $\|\psi_{0}\|_{2}^{2}<\|\nabla W\|_{2}^{2}$, see Theorem \ref{Theorem 5.1}. While, for $1+\frac{4}{D}\le p<1+\frac{4}{D-2}$, by combing the blow-up result in Theorem  \ref{Theorem 4.1}, the strong instability of standing waves with all positive frequency is derived, see Theorem \ref{Theorem 5.4}. Therefore, the  above stability results reveal the sharp criteria in the sense of the critical exponent $p=1+\frac{4}{D}$.  

\subsection{Variational characterization of ground states}
Substituting the standing wave $\psi(t, x)=u(x)e^{i\omega t}$ into Eq.~\eqref{Hartree equation} yields the following elliptic equation:
 \begin{equation}\label{elliptic}
	-\omega u+\Delta u+(|x|^{-2}*|u|^{2})u+|u|^{p-1}u=0. 
\end{equation}
Now, we introduce the Lagrange functional 
\begin{equation}\label{lagrange functional}
	\mathcal{L}[u]=\frac{1}{2}\mathcal{K}[u]-\frac{1}{p+1}\int_{\mathbb{R}^{D}} |u|^{p+1}dx-\frac{1}{4}\int_{\mathbb{R}^{D}} (|x|^{-2}*|u|^{2})|u|^{2}dx  
\end{equation}
and the Nehari functional
\begin{equation}\label{nehari}
	\mathcal{N}[u]=\mathcal{K}[u]-\int_{\mathbb{R}^{D}} |u|^{p+1}dx-\int_{\mathbb{R}^{D}} (|x|^{-2}*|u|^{2})|u|^{2}dx,  
\end{equation}
where $\mathcal{K}[u]=\int_{\mathbb{R}^{D}} |\nabla u|^{2}dx+\omega\int_{\mathbb{R}^{D}} |u|^{2}dx$. 
It follows from the Sobolev embedding inequality that above functionals are well-defined. 

Now, we investigate the  following manifold 
\begin{equation}
	\mathfrak{N}:=\{u\in H^{1}(\mathbb{R}^{D})\backslash  \{0\},\ \mathcal{N}[u]=0\} 
\end{equation}
and a variational problem on it defined by 
\begin{equation}\label{variational problem}
	d_{\mathfrak{N}}=\inf_{u\in\mathfrak{N}}\mathcal{L}[u]. 
\end{equation}
Indeed, this variational problem can be attained at some $u\in\mathfrak{N}$, and due to the variational structure of $\mathcal{L}[u]$ and $\mathcal{N}[u]$, the minimizer directly corresponds to  the nontrivial ground state solution of Eq. \eqref{elliptic}, these lead to the following result. 
\begin{prop}\label{noempty}
	Let $D\ge 3$ and $\omega>0$. If $1+\frac{4}{D}\le p<1+\frac{4}{D-2}$, then there exists $u\in \mathfrak{N}$ such that $\mathcal{L}[u]=d_{\mathfrak{N}}$, and $u$ is the ground state solution of Eq. \eqref{elliptic}.  
\end{prop}
\begin{rem}
	In subsection \ref{strong instability}, in the spirit of \cite{Zhang2005}, we establish the strong instability of standing waves for $1+\frac{4}{D}\le p<1+\frac{4}{D-2}$ and $\omega>0$, this result in Proposition \ref{noempty} is a crucial preliminary  in the proof of Theorem \ref{Theorem 5.4}.  
\end{rem}
Next, we introduce functional
\begin{equation}
	\mathcal{V}[u]=\int_{\mathbb{R}^{D}} |\nabla u|^{2}dx-\frac{D(p-1)}{2(p+1)}\int_{\mathbb{R}^{D}} |u|^{p+1}dx-\frac{1}{2}\int_{\mathbb{R}^{D}} (|x|^{-2}*|u|^{2})|u|^{2}dx. 
\end{equation}
Consider another manifold 
\begin{equation}\label{vpm}
	\mathfrak{M}:=\{u\in\mathbb{R}^{D},\ \mathcal{N}[u]<0,\ \mathcal{V}[u]=0\}  
\end{equation}
and variational problem on it defined by 
\begin{equation}
	d_{\mathfrak{M}}=\inf \{\mathcal{L}[u],\ u\in\mathfrak{M}\}.
\end{equation}
Actually, for all $\omega>0$, there holds a relationship naturally between $d_{\mathfrak{M}}$ and $d_{\mathfrak{N}}$, we state it as following proposition.
\begin{prop}\label{compare}
	Let $D\ge 3$ and $\omega>0$. If $1+\frac{4}{D}\le p<1+\frac{4}{D-2}$, then $d_{\mathfrak{N}}\le d_{\mathfrak{M}}$. 
\end{prop}
\begin{rem}
	 Propositions \ref{noempty} and \ref{compare} make detail discussions on variational problems \eqref{variational problem} and \eqref{vpm}, from the proof of these Theorems, the structures of functionals are of great impotence in understanding the variational characterizations, and our method is on the basis of profile decomposition theory, which will be introduced in next section. 
\end{rem}

\subsection{Main results}
This subsection is devoted to stating our main results. First of all, we define the cross constrained invariant manifolds as follows 
\begin{equation}
	\mathfrak{K}:=\{\psi\in H^{1}(\mathbb{R}^{D}), \mathcal{L}[\psi]<d_{\mathfrak{N}}, \mathcal{N}[\psi]<0, \mathcal{V}[\psi]<0\}, 
\end{equation}
\begin{equation}
	\mathfrak{K^{+}}:=\{\psi\in H^{1}(\mathbb{R}^{D}), \mathcal{L}[\psi]<d_{\mathfrak{N}}, \mathcal{N}[\psi]<0, \mathcal{V}[\psi]>0\}, 
\end{equation}
\begin{equation}
	\mathfrak{R^{+}}:=\{\psi\in H^{1}(\mathbb{R}^{D}), \mathcal{L}[\psi]<d_{\mathfrak{N}}, \mathcal{N}[\psi]>0\}. 
\end{equation}
Then, we can obtain the following sharp criteria of blowup for the Cauchy problem (\ref{Hartree equation})-(\ref{initial data}), as follows.
\begin{thm}\label{Theorem 4.1}Let $D\ge 3$, $1+\frac{4}{D}\le p<1+\frac{4}{D-2}$, and  $\omega>0$. 
	
	(i) If $\psi_{0}\in\mathfrak{K}$, then the corresponding solution of the Cauchy problem (\ref{Hartree equation})-(\ref{initial data}) blows up in a finite time. 
	
	(ii) If $\psi_{0}\in\mathfrak{K^{+}}\cup\mathfrak{R^{+}}$, then the corresponding solution of the Cauchy problem (\ref{Hartree equation})-(\ref{initial data}) exists globally for all time. 
\end{thm}
\begin{rem}\label{1.16}
	By some fundamental computations, one can see that following decomposition holds 
	\begin{equation}\label{4.03}
		\{\psi\in H^{1}(\mathbb{R}^{D}),\  \mathcal{L}[\psi]<d_{\mathfrak{N}}\}=\mathfrak{K}^{+}\cup\mathcal{R}^{+}\cup\mathfrak{K}, 
	\end{equation}
	then Theorem \ref{Theorem 4.1} reveals a sharp criteria for global existence and blow-up solutions.
\end{rem}
\begin{rem}
	There are some results for blow-up solutions of Cauchy problem (\ref{Hartree equation})-(\ref{initial data}). Fang et al. \cite{FH2011} (see also \cite{Cazenave2003}) proved the existence of blow-up solutions under assumptions $\mathscr{E}[\psi_0]<0$ and $1+\frac{4}{D}\le p\le 1+\frac{4}{D-2}$ in energy space $\Sigma=\{u\in H^{1}(\mathbb{R}^{D}), |x|u\in L^{2}(\mathbb{R}^{D})\}$. For $p=1+\frac{4}{D}$, Tian and Zhu \cite{TianZhu2022} derived the blow-up result for arbitrary supcritical mass by scaling invariance. When $p>1+\frac{4}{D}$, Leng et al. \cite{Leng2017} firstly give the sharp threshold for global existence and blow-up solutions by analyzing the structure of energy functional, and the main argument comes from \cite{ZhangZhu2011}, and has been widely employed to determine the sharp threshold of nonlinear Schr\"odinger equations (see for instance \cite{Feng2018,S2021,SA2022,TYZZ2021}). 
\end{rem}
\begin{rem}
	Besides, under the loss of scaling invariance and the interaction of nonlinear terms, we investigate the sharp criteria of Cauchy problem (\ref{Hartree equation})-(\ref{initial data}) for $1+\frac{4}{D}\le p<1+\frac{4}{D-2}$ by an alternative cross constrained variational method, which is different with previous results, and our results make a fundamental improvement in studying the strong instability of standing waves for all positive frequencies. Actually, for Eq.~(\ref{Hartree equation}), this method can give a exact comparison of two cross constrained variational problems, then for all $\omega>0$, we prove that  the standing waves are strongly unstable, thereby resolve the open problem proposed in \cite{Zhang2005}. 
\end{rem}
Then, we consider the sharp exponent of existence for stable  standing waves, which are not shown in previous studies. To start with, we give the definition of orbital stability. It follows from Cazenave and Lions' arguments \cite{CL1982} that the orbital stability corresponds with following variational problem 
\begin{equation}\label{2.5}
	d_{m}:=\inf_{u\in H^{1}(\mathbb{R}^{D}),\ \|u\|_{L^{2}(\mathbb{R}^{D})}^{2}=m}\mathscr{E}[u].  
\end{equation}
If the infimum of \eqref{2.5} can be attained, then we can define the minimizer set as follows
\begin{equation}
	\Omega:=\{u\in H^{1}(\mathbb{R}^{D}), \mathscr{E}[u]=d_{m}, \|u\|_{L^{2}(\mathbb{R}^{D})}^{2}=m\}. 
\end{equation}
\begin{defn}
	For any $\varepsilon>0$, if there exists $\delta>0$ such that 
	$
		\inf\limits_{u\in\Omega}\|\psi_{0}-u\|_{H^{1}(\mathbb{R}^{D})}<\delta$, 
	where $\psi_{0}$ is the initial datum, and the corresponding solution $\psi(t, x)$ for the Cauchy problem (\ref{Hartree equation})-(\ref{initial data}) with $\psi_{0}$ has following estimate
	\begin{equation*}
		\inf_{u\in\Omega}\|\psi(t, x)-u\|_{H^{1}(\mathbb{R}^{D})}<\varepsilon, \quad \text{for all }t>0, 
	\end{equation*}
	then the set $\Omega$ is called orbital stability. 
\end{defn}
By utilizing profile decomposition and scaling technique, we derive the orbital stability and strong instability of standing waves, as follows. 
\begin{thm}\label{Theorem 5.1}
	Let $D\ge 3$ and $1<p<1+\frac{4}{D}$. If  $\|\psi_{0}\|_{L^{2}(\mathbb{R}^{D})}^{2}<\|\nabla W\|_{L^{2}(\mathbb{R}^{D})}^{2}$, where $W$ is the ground state solution of (\ref{2.4}) and $\psi_{0}$ is the initial datum, then the standing waves of Eq. (\ref{Hartree equation}) is orbitally stable. 
\end{thm}
\begin{thm}\label{Theorem 5.4}
	Let $D\ge 3$ and $\omega>0$. If $1+\frac{4}{D}\le p<1+\frac{4}{D-2}$,  then the standing waves of Eq. (\ref{Hartree equation}) are strong instability in the following sense: for the minimizer $u$ of the variational problem \eqref{variational problem} and  for any $\varepsilon>0$, there exists initial datum $\psi_{0}$ such that $\|u-\psi_{0}\|_{H^{1}(\mathbb{R}^{D})}<\varepsilon$, then the corresponding solution $\psi(t, x)$ of the Cauchy problem (\ref{Hartree equation})-(\ref{initial data}) blows up in a finite time. 
\end{thm}
\begin{rem}
	For all $\omega>0$, Wang and Zhang \cite{WZ2025} also obtained the strong instability of standing waves for Eq. \eqref{Hartree equation}, the main method is standard variational arguments, but they only consider the case $D=3$ and $3\le p<5$, this result is contained and extended to more general case in Theorem \ref{Theorem 5.4}. 
	\end{rem}
\begin{rem}
It follows from Theorems \ref{Theorem 5.1} and \ref{Theorem 5.4} that the critical exponent $p = 1 + \frac{4}{D}$ distinguishes between stable and strongly unstable standing waves, thus we obtain the sharp results for stability. 
\end{rem}

\subsection{Outline of this paper}
The rest of this paper is organized as follows. In Section \ref{Preliminaries}, two kinds of sharp Gagliardo-Nirenberg inequalities and profile decomposition theory are introduced. In Section \ref{Variarional Characterization}, the variational characterizations of ground states are investigated. The sharp criteria of blowup is studied in Section \ref{section4}, and Theorem \ref{Theorem 4.1} is proved by cross-constrained variational method. Section \ref{section5} reveals the sharp criteria for stability of standing waves, and Theorems \ref{Theorem 5.1} and \ref{Theorem 5.4} are derived.
For convenience, the following notations are utilized in the rest of this paper: $\|\cdot\|_{p}:=\|\cdot\|_{L^{p}(\mathbb{R}^{D})}$ and  $\int\cdot dx:=\int_{\mathbb{R}^{D}}\cdot dx$. 

\section{Preliminaries}\label{Preliminaries}
In this section, we will recall some crucial estimates and variational tools, which have been widely used in studying the nonlinear Schr\"odinger equations. Firstly, let's consider the best constant of Gagliardo-Nirenberg (GN) inequality, which plays a pivotal role in understanding the dynamical behaviors of critical nonlinear Schr\"odinger equations, including sharp mass threshold \cite{Zhu2016}, orbital stability of standing waves \cite{ZhangZhuJDDE}. Now, we introduce following two Lemmas regarding with sharp GN inequality. 
\begin{lem}(\cite{Cazenave2003,Weinstein1983})
	Let $D\ge 3$ and $0<q<\frac{2}{D-2}$, there holds following estimate 
	\begin{equation}\label{GN}
		\|v\|_{2q+2}^{2q+2}\le C_{\text{opt}}\|\nabla v\|_{2}^{q D}\|v\|_{2}^{2+q(2-D)}, 
	\end{equation}
	where $C_{\text{opt}}=\frac{q+1}{\|R\|_{2}^{2q}}$, and $R$ solves the following elliptic equation
	\begin{equation}\label{Q}
		\frac{q N}{2}\Delta R-\Big(1+\frac{q}{2}(2-D)\Big)R+R^{2q+1}=0, \ \ R\in H^{1}(\mathbb{R}^{D}). 
	\end{equation}
\end{lem}
For the Hardy nonlinear term, the last author of this present paper obtained the sharp generalized Gagliardo–Nirenberg (gGN) inequality in $H^{1}(\mathbb{R}^{D})$. 
\begin{lem}(\cite{Zhu2016})
	Let $D\ge 3$. One admits that 
	\begin{equation}\label{GGN}
		\|(|x|^{-2}*|v|^{2})|v|^{2}\|_{1}\le C_{\text{opt}}\|v\|_{2}^{2}\|\nabla v\|_{2}^{2}, 
	\end{equation}
	where $C_{\text{opt}}=\frac{2}{\|\nabla W\|_{2}^{2}}$ and $W$ satisfies the following elliptic equation 
	\begin{equation}\label{2.4}
		-\Delta W+W-(|x|^{-2}*|W|^{2})W=0, \ \ W\in H^{1}(\mathbb{R}^{D}). 
	\end{equation}
\end{lem}
Next, we introduce profile decomposition theory, a useful variational tool which is proposed by Gérard \cite{G1998}, Hmidi and Keraani \cite{HK2005}.  
\begin{prop}\label{prop 5.2}
	Let $D\ge 3$ and $\{u_{n}\}_{n=1}^{+\infty}$ be a bounded sequence in $H^{1}(\mathbb{R}^{D})$, then there exists a subsequence of $\{u_{n}\}_{n=1}^{+\infty}$ (still denoted by $\{u_{n}\}_{n=1}^{+\infty}$), a family $\{x_{n}^{j}\}_{j=1}^{+\infty}$ of sequences in $\mathbb{R}^{D}$ and a family $\{U^{j}\}_{j=1}^{+\infty}$ of functions in $H^{1}(\mathbb{R}^{D})$ fulfilling following properties. 
	\item[(i)] For any $k_{1}\neq k_{2}$
	\begin{equation*}
		|x_{n}^{k_{1}}-x_{n}^{k_{2}}|\to +\infty, \quad\text{as}\quad n\to+\infty. 
	\end{equation*}
	\item[(ii)] For any $J\ge 1$ and $x\in\mathbb{R}^{D}$, $u_{n}$ can be decomposed as 
	\begin{equation}\label{DP}
		u_{n}(x)=\sum_{j=1}^{J}U^{j}(x-x_{n}^{j})+u_{n}^{J}(x),
	\end{equation}
	where 
	\begin{equation*}
		\lim_{J\to+\infty}\limsup_{n\to+\infty}\|u_{n}^{J}\|_{p}=0, \quad\text{for every}\quad p\in(2, \frac{2D}{D-2}), 
	\end{equation*}
	Moreover, we have following estimate. 
	\begin{equation}\label{2.51}
		\|u_{n}\|_{2}^{2}=\sum_{j=1}^{J}\|U^{j}\|_{2}^{2}+\|u_{n}^{J}\|_{2}^{2}+o(1), 
	\end{equation}
	\begin{equation}\label{2.61}
		\|\nabla u_{n}\|_{2}^{2}=\sum_{j=1}^{J}\|\nabla U^{j}\|_{2}^{2}+\|u_{n}^{J}\|_{2}^{2}+o(1), 
	\end{equation}
	\begin{equation}\label{2.71}
		\|u_{n}\|_{q+1}^{q+1}=\sum_{j=1}^{J}\|U^{j}\|_{q+1}^{q+1}+\|u_{n}^{J}\|_{q+1}^{q+1}+o(1), 
	\end{equation}
		\begin{equation}\label{2.81}
		\|(|x|^{-2}*|u_{n}|^{2})|u_{n}|^{2}\|_{1}=\sum_{j=1}^{J}\|(|x|^{-2}*|U^{j}|^{2})|U^{j}|^{2}\|_{1}+\|(|x|^{-2}*|u_{n}^{J}|^{2})|u_{n}^{J}|^{2}\|_{1}+o(1), 
	\end{equation}
	where $1<q<1+\frac{4}{D-2}$ and $o(1)=o_{n}(1)\to 0$ as $n\to+\infty$. 
\end{prop}

\begin{rem}
We should point out that  the decomposition of  $\|(|x|^{-2}*|u_{n}|^{2})|u_{n}|^{2}\|_{1}$ is obtained by the last author of this paper in \cite{ZhangZhuJDDE}, thus we just omit the proof here.  On the other hand, about the profile decomposition theory, we have the following comments. Proposition \ref{prop 5.2} is another representation of the 
concentration-compactness principle proposed by Lions in \cite{Lions1984}. Actually, 
 the number of  components in   (\ref{2.5}) may be one, finite and infinite, which are corresponding to compactness, dichotomy and vanishing in the concentration-compactness principle. 
 However, there are two major advantages of  Proposition \ref{prop 5.2}: one is
  that (\ref{DP}) provides a  exact  representation of the bounded sequence, which can be applied in the corresponding functionals. The other is orthogonality of the decomposition, which can be used to simplify the calculus of variational problem.
 \end{rem}

\section{Variarional Characterization}\label{Variarional Characterization}
In this section, by utilizing scaling technique and profile decomposition theory, we investigate the variartional characterizations of Eq. \eqref{Hartree equation}, and derive the comparison of the infimum of two variational problems \eqref{variational problem} and \eqref{vpm}. To start with, we consider the variational problem \eqref{variational problem}. 
\begin{prop}\label{d>0}
	Let $D\ge 3$ and $\omega>0$. If $1+\frac{4}{D}\le p<1+\frac{4}{D-2}$, then $d_{\mathfrak{N}}>0$. 
\end{prop}
\begin{proof}
	For any $u\in\mathfrak{N}$, one has $\mathcal{N}[u]=0$. It follows from the Sobolev embedding inequality and gGN inequality that 
	\begin{equation*}
		\begin{split}
			\mathcal{K}[u]&=\int |u|^{p+1}dx+\int (|x|^{-2}*|u|^{2})|u|^{2}dx\\
			&\le C_{1}\mathcal{K}[u]^{\frac{p+1}{2}}+C_{2}\int |u|^{2}dx\cdot\int |\nabla u|^{2}dx\\
			&\le C_{1}\mathcal{K}[u]^{\frac{p+1}{2}}+C_{3}\mathcal{K}[u]^{2}, 
		\end{split}
	\end{equation*}
	where $C_{j} \  (j=1, 2, 3)$ denote positive constant. Thus, one deduces that $\mathcal{K}[u]\ge C>0$. From $\mathcal{N}[u]=0$, $\mathcal{L}[u]$ can be rewritten as follows
	\begin{equation*}
		\begin{split}
			\mathcal{L}[u]&=(\frac{1}{2}-\frac{1}{p+1})\int |u|^{p+1}dx+\frac{1}{4}\int (|x|^{-2}*|u|^{2})|u|^{2}dx\\
			&\ge C_{\min}\mathcal{K}[u]\\
			&>0, 
		\end{split}
	\end{equation*}
	where $C_{\min}=\min\{\frac{1}{2}-\frac{1}{p+1}, \frac{1}{4}\}$. 
	Then, it follows from \eqref{variational problem} that $d_{\mathfrak{N}}>0$. 
	
\end{proof}
Building on the result in Proposition~\ref{d>0}, we can establish  Proposition~\ref{noempty}.
\begin{proof}{\it (The proof of Proposition \ref{noempty}) }Firstly, we will define some non-negative functionals, which is crucial to our proof. 
	\begin{equation}
		\mathcal{L}_{1}[u]=\frac{1}{4}\mathcal{K}[u]+(\frac{1}{4}-\frac{1}{p+1})\int |u|^{p+1}dx, 
	\end{equation}
	\begin{equation}
		\mathcal{L}_{2}[u]=(\frac{1}{2}-\frac{1}{p+1})\mathcal{K}[u]+(\frac{1}{p+1}-\frac{1}{4})\int (|x|^{-2}*|u|^{2})|u|^{2}dx. 
	\end{equation}
	Actually, with different range of $p$, one can rewrite $\mathcal{L}[u]$ as follows. 
	\begin{equation}\label{3.10}
		\begin{split}
		\mathcal{L}[u]&=\frac{1}{4}\mathcal{N}[u]+\mathcal{L}_{1}[u], \qquad (3\le p<1+\frac{4}{D-2}),\\
			&=\frac{1}{p+1}\mathcal{N}[u]+\mathcal{L}_{2}[u], \qquad (1+\frac{4}{D}\le p<3). 
		\end{split}
	\end{equation}
	It follows from \eqref{3.10} that,  the variational problem \eqref{variational problem} is changed to 
	\begin{equation}\label{variational problem 2}
		d_{\mathfrak{N}}=\inf_{u\in\mathfrak{N}}\mathcal{L}_{k}[u], 
	\end{equation}
	where $k=1$ for $3\le p<1+\frac{4}{D-2}$ and $k=2$ for $1+\frac{4}{D}\le p<3$. 
	
	Let $\{u_{n}\}_{n=1}^{+\infty}$ be the minimizing sequence of the variational problem \eqref{variational problem 2}, then there holds 
	\begin{equation*}
		\mathcal{N}[u_{n}]=0,\quad \mathcal{L}_{k}[u_{n}]\to d_{\mathfrak{N}}\quad\text{as}\quad n\to+\infty. 
	\end{equation*}
	It follows from the definition of $\mathcal{L}_{k}[u]$ that $\{u_{n}\}_{n=1}^{+\infty}$ is bounded in $H^{1}(\mathbb{R}^{D})$, then $u_{n}$ can be decomposed as follows. 
	\begin{equation*}
		u_{n}(x)=\sum_{j=1}^{J}U^{j}(x-x_{n}^{j})+u_{n}^{J}(x),
	\end{equation*}
	Moreover, estimates \eqref{2.51}-\eqref{2.81} hold, and  one has 
	\begin{equation*}
		\mathcal{N}[u_{n}]=\sum_{j=1}^{J}\mathcal{N}[U^{j}]+\mathcal{N}[u_{n}^{J}]+o(1)
	\end{equation*}
	and
	\begin{equation*}
		\mathcal{L}_{k}[u_{n}]=\sum_{j=1}^{J}\mathcal{L}_{k}[U^{j}]+\mathcal{L}_{k}[u_{n}^{J}]+o(1), 
	\end{equation*}
	where $o(1)=o_{n}(1)\to 0$ as $n\to+\infty$. 
	
	Here, we claim that $\mathcal{N}[U^{j}]\ge 0$ for all $1\le j\le J$. If not, there exists $1\le j_{0}\le J$ such that $\mathcal{N}[U^{j_{0}}]<0$, then one can select suitable $0<\lambda<1$ such that $\mathcal{N}[\lambda U^{j_{0}}]=0$, thus $\mathcal{L}_{k}[\lambda U^{j_{0}}]\ge d_{\mathfrak{N}}$ by $\lambda U^{j_{0}}\in\mathfrak{N}$. Besides, following estimate holds 
	\begin{equation}
		\begin{split}
			\lambda^{2}\mathcal{L}_{k}[U^{j_{0}}]>\mathcal{L}_{k}[\lambda U^{j_{0}}]\ge d_{\mathfrak{N}}\ge \mathcal{L}_{k}[U^{j_{0}}], 
		\end{split}
	\end{equation}
	here, the last inequality follows from decomposition of $\mathcal{L}_{k}[u_{n}]$. Clearly, the contradiction appears, then $\mathcal{N}[U^{j}]\ge 0$ for all $1\le j\le J$. And the decomposition of $\mathcal{N}[u_{n}]$ also shows that $\mathcal{N}[U^{j}]=0$ for all $1\le j\le J$. 
	
	Besides, there exists a unique $1\le j_{1}\le J$ such that $\mathcal{N}[U^{j_{1}}]=0$ and $\mathcal{L}_{k}[U^{j_{1}}]=d_{\mathfrak{N}}$. Indeed, from $\mathcal{N}[U^{j_{1}}]=0$ and decomposition of $\mathcal{L}_{k}[u_{n}]$, one has
	\begin{equation*}
		\mathcal{L}_{k}[U^{j_{1}}]\ge d_{\mathfrak{N}}\ge \mathcal{L}_{k}[U^{j_{1}}], 
	\end{equation*}
	this estimate also shows the uniqueness. Then we finish the first part of this proof. 
	
	If $u\in\mathcal{O}$, then $\mathcal{N}[u]=0$, and 
		\begin{equation}
		\begin{split}
			<\mathcal{N}'[u], u>&=2\int |\nabla u|^{2}dx+2\omega\int |u|^{2}dx-(p+1)\int |u|^{p+1}dx\\
			&\quad-4\int (|x|^{-2}*|u|^{2})|u|^{2}dx\\
			&=(1-p)\int |u|^{p+1}dx-2\int (|x|^{-2}*|u|^{2})|u|^{2}dx\\
			&<0. 
		\end{split}
	\end{equation}
	Since $u$ is the minimizer of the variational problem \eqref{variational problem}, thus there exists a lagrange multiplier $\Lambda\in \mathbb{R}$ such that 
	\begin{equation}
		\mathcal{L}'[u]+\Lambda \mathcal{N}'[u]=0, 
	\end{equation}
	then
	\begin{equation}
		<\mathcal{L}'[u]+\Lambda \mathcal{N}'[u], u>=\mathcal{N}[u]+\Lambda<\mathcal{N}'[u], u>=0, 
	\end{equation}
	which implies $\Lambda=0$, then $\mathcal{L}'[u]=0$. For any ground state solution $v$ of Eq. \eqref{Hartree equation}, one has $\mathcal{L}'[v]=0$, then $\mathcal{N}[v]=0$, $v\in\mathfrak{N}$. It follows that
	\begin{equation}
		\mathcal{L}[u]=d_{\mathfrak{N}}\le \mathcal{L}[v], 
	\end{equation}
	which implies that $u$ is a ground state solution of elliptic equation \eqref{elliptic}. 
	
\end{proof}

Naturally, one can obtain following proposition. 
\begin{prop}\label{prop 3.3}
	Let $D\ge 3$ and $\omega>0$. There exists $u\in H^{1}(\mathbb{R}^{D})\setminus \{0\}$ such that $\mathcal{N}[u]=0$ and $\mathcal{V}[u]=0$. 
\end{prop}
\begin{proof}
	
 It follows from Proposition \ref{noempty} that there exists a nontrivial $u$ such that $u$ is a ground state solution of Eq. \eqref{elliptic}, which indicates that $\mathcal{N}[u]=0$. Mutiplying Eq. (\ref{elliptic}) with $x\cdot\nabla \bar{u}$, one can obtain the Pohozaev identity, which is a linear combination of $\mathcal{N}[u]$ and $\mathcal{V}[u]$ as follows. 
	\begin{equation*}
		\mathcal{N}[u]-\frac{2}{D}\mathcal{V}[u]=0, 
	\end{equation*}
	thus, one has $\mathcal{V}[u]=0$ from $\mathcal{N}[u]=0$. 
	
\end{proof}

\begin{prop}
	Let $D\ge 3$ and $\omega>0$. If $1+\frac{4}{D}\le p<1+\frac{4}{D-2}$, then $\mathfrak{M}$ is not empty. 
\end{prop}
\begin{proof}
	It follows from Proposition \ref{prop 3.3} that $\mathcal{N}[u]=0$ and $\mathcal{V}[u]=0$ for $u\in H^{1}(\mathbb{R}^{D}) \setminus \{0\}$. We will proceed this proof with two cases. 
	
	\textbf{Case (i). }$1+\frac{4}{D}<p<1+\frac{4}{D-2}$.
	For any $\lambda>1$, one has $\mathcal{N}[\lambda u]<0$ and $\mathcal{V}[\lambda u]<0$. By denoting $v=\lambda u$ here, and  selecting suitable $\lambda$, one has following estimates
	\begin{equation}
		\int |\nabla v|^{2}dx-\int |v|^{p+1}dx<0, 
	\end{equation}
		\begin{equation}\label{3.26}
		\int |\nabla v|^{2}dx-\frac{D(p-1)}{2(p+1)}\int |v|^{p+1}dx>0. 
	\end{equation}
	Consider following re-scaling $v_{\lambda}=\lambda^{\frac{1}{p-2}}v(\lambda x)$.
	\begin{equation}
		\begin{split}
		\mathcal{N}[v_{\lambda}]&=\lambda^{\frac{2(p+1)}{p-1}-D}\Big(\int |\nabla v|^{2}dx-\int |u|^{p+1}dx\Big)+\omega\lambda^{\frac{4}{p-1}-D}\int |u|^{2}dx\\
		&-\lambda^{\frac{8}{p-1}-2D+2}\int (|x|^{-2}*|u|^{2})|u|^{2}dx, 
		\end{split}
	\end{equation}
	\begin{equation}
		\begin{split}
		\mathcal{V}[v_{\lambda}]&=\lambda^{\frac{2(p+1)}{p-1}-D}\Big(\int |\nabla v|^{2}dx-\frac{D(p-1)}{2(p+1)}\int |u|^{p+1}dx\Big)\\
		&-\frac{1}{2}\lambda^{\frac{8}{p-1}-2D+2}\int (|x|^{-2}*|u|^{2})|u|^{2}dx. 
		\end{split}
	\end{equation}
	One can see that $\frac{2(p+1)}{p-1}-D>0$ and  $\frac{2(p+1)}{p-1}-D>\frac{8}{p-1}-2D+2>\frac{4}{p-1}-D$. Thus, from \eqref{3.26}, one has  $\lim\limits_{\lambda\to+\infty}\mathcal{V}[v_{\lambda}]=+\infty$, and there exists $\lambda^{*}>1$ such that $\mathcal{V}[u_{\lambda^{*}}]=0$. Meanwhile, $\mathcal{N}[v_{\lambda^{*}}]$ conserve negative, which implies $v_{\lambda^{*}}\in\mathfrak{M}$. 
	
	\textbf{Case (ii). }$p=1+\frac{4}{D}$. In this case, we start from $\mathcal{N}[u]=0$ and $\mathcal{V}[u]=0$. Define the scaling transformation $u_{\lambda}=\lambda^{\frac{D}{2}}u(\lambda x)$. 
	\begin{equation}
		\begin{split}
				\mathcal{N}[u_{\lambda}]&=\lambda^{2}\int |\nabla v|^{2}dx+\omega\int |u|^{2}dx-\lambda^{2}\int |u|^{\frac{2N+4}{N}}dx\\
				&-\lambda^{2}\int (|x|^{-2}*|u|^{2})|u|^{2}dx, 
		\end{split}
	\end{equation}
	\begin{equation}
		\mathcal{V}[u_{\lambda}]=\lambda^{2}\mathcal{V}[u]=0, 
	\end{equation}
Then, $\lambda>1$ implies  $\mathcal{N}[u_{\lambda}]<0$, and  $u_{\lambda}\in\mathfrak{M}$ for $p=1+\frac{4}{D}$. 
	
\end{proof}

\begin{prop}
	Let $D\ge 3$ and $\omega>0$. If $1+\frac{4}{D}\le p<1+\frac{4}{D-2}$, then $d_{\mathfrak{M}}>0$. 
\end{prop}
\begin{proof}
	Let $u\in\mathfrak{M}$, then $\mathcal{N}[u]<0$ and $\mathcal{V}[u]=0$. And $\mathcal{L}[u]$ can be rewritten as 
	\begin{equation}
		\mathcal{L}[u]=\frac{\omega}{2}\int |u|^{2}dx+\frac{1}{p+1}\Big[\frac{N(p-1)}{4}-1\Big]\int |u|^{p+1}dx\ge 0, 
	\end{equation} 
then, from \eqref{vpm}, one has $d_{\mathfrak{M}}\ge 0$. Assume by contradiction, if $d_{\mathfrak{M}}=0$, let $\{u_{n}\}_{n=1}^{+\infty}\subset\mathfrak{M}$ be the sequence such that $\mathcal{N}[u_{n}]<0$, $\mathcal{V}[u_{n}]=0$, and $\mathcal{L}[u_{n}]\to 0$ as $n\to+\infty$. 

\textbf{Case (i). }When $1+\frac{4}{D}<p<1+\frac{4}{D-2}$, one sees that  $\int |u_{n}|^{2}dx\to 0$ and $\int |u_{n}|^{p+1}dx\to 0$  as $n\to+\infty$. By $\mathcal{N}[u_{n}]<0$, one derives
\begin{equation}
	\begin{split}
			\mathcal{K}[u_{n}]&<\int |u_{n}|^{p+1}dx+\int (|x|^{-2}*|u_{n}|^{2})|u_{n}|^{2}dx\\
			&\le \int |u_{n}|^{p+1}dx+C\int |\nabla u_{n}|^{2}dx\cdot\int |u_{n}|^{2}dx, 
	\end{split}
\end{equation}
where $C$ denotes the positive constant. However, by letting $n\to+\infty$, one obtains
\begin{equation}
	\mathcal{K}[u_{n}]\ge \int |u_{n}|^{p+1}dx+C\int |\nabla u_{n}|^{2}dx\cdot\int |u_{n}|^{2}dx, 
\end{equation}
thus $d_{\mathfrak{M}}>0$ for $1+\frac{4}{D}<p<1+\frac{4}{D-2}$. 

\textbf{Case (ii). }When $p=1+\frac{4}{D}$, it follows that  $\int |u_{n}|^{2}dx\to 0$ as $n\to+\infty$. Similarly, from $\mathcal{N}[u_{n}]<0$, one has 
\begin{equation}
	\begin{split}
		\mathcal{K}[u_{n}]&<\int |u_{n}|^{p+1}dx+\int (|x|^{-2}*|u_{n}|^{2})|u_{n}|^{2}dx\\
		&\le C_{1}(\int |\nabla u_{n}|^{2}dx)^{\frac{D(p-1)}{4}}\cdot(\int |u_{n}|^{2}dx)^{\frac{p+1}{2}-\frac{D(p-1)}{4}}\\
		&\quad+C_{2}\int |\nabla u_{n}|^{2}dx\cdot\int |u_{n}|^{2}dx, 
	\end{split}
\end{equation}
where $C_{1}$ and $C_{2}$ denote the positive constant. However, when $n\to+\infty$, one obtains
\begin{equation}
	\begin{split}
			\mathcal{K}[u_{n}]&\ge C_{1}(\int |\nabla u_{n}|^{2}dx)^{\frac{D(p-1)}{4}}\cdot(\int |u_{n}|^{2}dx)^{\frac{p+1}{2}-\frac{D(p-1)}{4}}\\
			&\quad+C_{2}\int |\nabla u_{n}|^{2}dx\cdot\int |u_{n}|^{2}dx, 
	\end{split}
\end{equation}
which leads a  contradiction. Hence,  $d_{\mathfrak{M}}>0$ in this case. 

\end{proof}

Now, we have derived some important properties for  $d_{\mathfrak{N}}$ and  $d_{\mathfrak{M}}$, which can be used to prove that $d_{\mathfrak{N}}\le d_{\mathfrak{M}}$ holds naturally for all $\omega>0$. 
\begin{proof}{\it (The proof of Proposition \ref{compare})}
	Define the scaling transformation  $u_{\lambda}=\lambda^{\frac{D}{2}}u(\lambda x)$. It follows that 
	\begin{equation}
		\begin{split}
			\mathcal{L}[u_{\lambda}]&=\frac{1}{2}\lambda^{2}\int |\nabla u|^{2}dx+\frac{\omega}{2}\int |u|^{2}dx-\frac{1}{p+1}\lambda^{\frac{D(p-1)}{2}}\int |u|^{p+1}dx\\
			&\quad-\frac{1}{4}\lambda^{2}\int (|x|^{-2}*|u|^{2})|u|^{2}dx, 
		\end{split}
	\end{equation}
	\begin{equation}
		\begin{split}
			\mathcal{N}[u_{\lambda}]&=\lambda^{2}\int |\nabla u|^{2}dx+\omega\int |u|^{2}dx-\lambda^{\frac{D(p-1)}{2}}\int |u|^{p+1}dx\\
			&\quad-\lambda^{2}\int (|x|^{-2}*|u|^{2})|u|^{2}dx, 
		\end{split}
	\end{equation}
	\begin{equation}
		\begin{split}
			\mathcal{V}[u_{\lambda}]&=\lambda^{2}\int |\nabla u|^{2}dx-\frac{D(p-1)}{2(p+1)}\lambda^{\frac{D(p-1)}{2}}\int |u|^{p+1}dx\\
			&\quad-\frac{1}{2}\lambda^{2}\int (|x|^{-2}*|u|^{2})|u|^{2}dx.
		\end{split}
	\end{equation}
After some careful checking, one can find  that 
	\begin{equation}\label{3.28}
		\frac{d}{d\lambda}\mathcal{L}[u_{\lambda}]=\frac{\mathcal{V}[u_{\lambda}]}{\lambda}. 
	\end{equation}
	Let $u\in\mathfrak{M}$, then $\mathcal{N}[u]<0$ and $\mathcal{V}[u]=0$. There exists $\lambda^{*}\in(0, 1)$ such that $\mathcal{N}[u_{\lambda^{*}}]=0$, then $\mathcal{L}[u_{\lambda^{*}}]\ge d_{\mathfrak{N}}$ by $u_{\lambda^{*}}\in\mathfrak{N}$. 
	
	From the structure of $\mathcal{V}[u_{\lambda}]$, the following results hold
	\begin{itemize}
		\item[(i)] $p=1+\frac{4}{D}$. From \eqref{3.28} and $\mathcal{V}[u]=0$, one sees that $\frac{d}{d\lambda}\mathcal{L}[u_{\lambda}]=\lambda^{2}\mathcal{V}[u]=0$, which implies $\mathcal{L}[u]=\mathcal{L}[u_{\lambda^{*}}]$. 
		\item[(ii)] $1+\frac{4}{D}<p<1+\frac{4}{D-2}$. There exists a unique $\lambda_{0}$ such that $\mathcal{V}[u_{\lambda_{0}}]=0$, and from \eqref{3.28},  $\lambda_{0}$ is the maximum point of $\mathcal{L}[u_{\lambda}]$, and $\mathcal{V}[u]=0$ implies $\lambda_{0}=1$. 
	\end{itemize}
	Thus, from above results and $\mathcal{V}[u]=0$, one has $d_{\mathfrak{N}}\le \mathcal{L}[u_{\lambda^{*}}]\le \mathcal{L}[u]$. Finally, by the arbitrariness of $u\in\mathfrak{M}$, one obtains $d_{\mathfrak{N}}\le d_{\mathfrak{M}}$ naturally. 
	
\end{proof}

\section{Sharp Criteria of Blowup}\label{section4}

In this section, we investigate the  sharp criteria of global existence and blowup  for Eq. \eqref{Hartree equation} by constructing invariant manifolds, as follows.
\begin{prop}\label{invariant}
	Let $D\ge 3$ and $\omega>0$. If $1+\frac{4}{D}\le p<1+\frac{4}{D-2}$ and $\omega>0$, then $\mathfrak{K}$, $\mathfrak{K^{+}}$, and $\mathfrak{R^{+}}$ are invariant manifolds of the Cauchy problem (\ref{Hartree equation})-(\ref{initial data}). 
\end{prop}
\begin{proof}
The proof of Proposition \ref{invariant} is a direct consequence of the conservation laws and the variational definitions by the continuity argument, and it is quite similar to that in \cite{Zhang2005}. Here, we omit the  detailed proof. 
	
\end{proof}
Now, let's recall the virial identity. Suppose that  $D\ge 3$,  $\psi_{0}\in H^{1}(\mathbb{R}^{D})$, $|x|\psi_{0}\in L^{2}(\mathbb{R}^{D})$ and $\psi(t, x)$ is the corresponding solution of the Cauchy problem (\ref{Hartree equation})-(\ref{initial data}) on $[0, T)$. Denote $\mathcal{G}(t)=\int|x|^{2}|\psi|^{2}dx$. Then, for all $t\in[0, T)$, 
\begin{equation}
	\mathcal{G}'(t)=4Im\int x\bar{\psi}\nabla\psi dx, 
\end{equation}
\begin{equation}\label{virial}
	\begin{split}
		\mathcal{G}''(t)&=16E[\psi]+\frac{16-4D(p-1)}{p+1}\int |\psi|^{p+1}dx. 
	\end{split}
\end{equation}
We finish the proof of Theorem \ref{Theorem 4.1}. 
\begin{proof}{\it (The proof of Theorem \ref{Theorem 4.1})} Firstly, when $\psi_{0}\in\mathfrak{K}$, it follows from Proposition \ref{invariant} that $\psi(t, \cdot)\in\mathfrak{K}$ for $t\in[0, T)$. Fix $t\in[0, T)$ and $\psi(t, \cdot)=\psi$, one has $\mathcal{L}[\psi]<d_{\mathfrak{N}}$, $\mathcal{N}[\psi]<0$, and $\mathcal{V}[\psi]<0$. Notice that the following relation holds. 
	\begin{equation}
		\mathcal{G}''(t)=8\mathcal{V}[\psi]. 
	\end{equation}
	Considering the scaling $\psi_{\lambda}=\lambda^{\frac{D}{p+1}}\psi(\lambda x)$,  one has
	\begin{equation}
		\begin{split}
		\mathcal{L}[\psi_{\lambda}]&=\frac{1}{2}\lambda^{\frac{2D}{p+1}+2-D}\int |\nabla\psi|^{2}dx+\frac{\omega}{2}\lambda^{\frac{2D}{p+1}-D}\int |\psi|^{2}dx-\frac{1}{p+1}\int |\psi|^{p+1}dx\\
		&\quad-\frac{1}{4}\lambda^{\frac{4D}{p+1}-2D+2}\int (|x|^{-2}*|\psi|^{2})|\psi|^{2}dx, 
		\end{split}
	\end{equation}
	\begin{equation}
		\begin{split}
		\mathcal{N}[\psi_{\lambda}]&=\lambda^{\frac{2D}{p+1}+2-D}\int |\nabla\psi|^{2}dx+\lambda^{\frac{2N}{p+1}-D}\int |\psi|^{2}dx-\int |\psi|^{p+1}dx\\
		&\quad-\lambda^{\frac{4D}{p+1}-2D+2}\int (|x|^{-2}*|\psi|^{2})|\psi|^{2}dx, 
		\end{split}
	\end{equation}
	\begin{equation}
		\begin{split}
			\mathcal{V}[\psi_{\lambda}]&=\lambda^{\frac{2D}{p+1}+2-D}\int |\nabla\psi|^{2}dx-\frac{D(p-1)}{2(p+1)}\int |\psi|^{p+1}dx\\
			&\quad-\frac{1}{2}\lambda^{\frac{4D}{p+1}-2N+2}\int (|x|^{-2}*|\psi|^{2})|\psi|^{2}dx.
		\end{split}
    \end{equation}
	Note that the  following relationship holds naturally, 
	\begin{equation*}
		\frac{2D}{p+1}+2-D>\frac{4D}{p+1}-2D+2>\frac{2D}{p+1}-D, \quad\text{and}\quad \frac{2D}{p+1}+2-D>0.
	\end{equation*}
which indicates that  $\lim\limits_{\lambda\to+\infty}\mathcal{V}[\psi_{\lambda}]=+\infty$. By the continuity of $\lambda$, one can find a $\lambda^{*}>1$ such that $\mathcal{V}[\psi_{\lambda^{*}}]=0$. Moreover, $\mathcal{V}[\psi_{\lambda}]<0$ for any $\lambda\in[1, \lambda^{*})$. We now distinguish cases based on the sign of    $\mathcal{N}[\psi_{\lambda}]$. 
	\begin{itemize}
		\item[(i)] $\forall$\ $\lambda\in[1, \lambda^{*}]$, one has $\mathcal{N}[\psi_{\lambda}]<0$; 
		\item[(ii)] $\exists$ $\mu\in(1, \lambda^{*}]$ such that $\mathcal{N}[\psi_{\mu}]=0$. 
	\end{itemize}
	For (i), one has $\mathcal{N}[\psi_{\lambda^{*}}]<0$ and $\mathcal{V}[\psi_{\lambda^{*}}]=0$, which implies $\psi_{\lambda^{*}}\in\mathfrak{M}$, then $\mathcal{L}[\psi_{\lambda^{*}}]\ge d_{\mathfrak{M}}\ge d_{\mathfrak{N}}>\mathcal{L}[\psi]$. After some computations, one derives
	\begin{equation}
		\begin{split}
			\mathcal{L}[\psi]-\mathcal{L}[\psi_{\lambda^{*}}]&=\frac{1}{2}(1-\lambda^{*\frac{2D}{p+1}+2-D})\int |\nabla\psi|^{2}dx+\frac{\omega}{2}(1-\lambda^{*\frac{2D}{p+1}-D})\int |\psi|^{2}dx\\
			&\quad-\frac{1}{4}(1-\lambda^{*\frac{4D}{p+1}+2-2D})\int (|x|^{-2}*|\psi|^{2})|\psi|^{2}dx\\
			&\ge \frac{1}{2}(\mathcal{V}[\psi]-\mathcal{V}[\psi_{\lambda^{*}}])\\
			&=\frac{1}{2}\mathcal{V}[\psi], 
		\end{split}
	\end{equation}
	and then
	\begin{equation}\label{3.3}
		\mathcal{V}[\psi]\le 2(\mathcal{L}[\psi]-\mathcal{L}[\psi_{\lambda^{*}}])\le 2(\mathcal{L}[\psi_{0}]-d_{\mathfrak{N}})<0.  
	\end{equation}
	
	For (ii), $\mathcal{N}[\psi_{\mu}]=0$ implies $\psi_{\mu}\in\mathfrak{N}$, then $\mathcal{L}[\psi_{\mu}]\ge d_{\mathfrak{N}}>\mathcal{L}[\psi]$, similar argument as in (i), one has
	\begin{equation}
		\mathcal{L}[\psi]-\mathcal{L}[\psi_{\mu}]\ge \frac{1}{2}(\mathcal{V}[\psi]-\mathcal{V}[\psi_{\mu}])\ge \frac{1}{2}\mathcal{V}[\psi].
	\end{equation}
It follows that 
	\begin{equation}\label{3.2}
		\mathcal{V}[\psi]\le 2(\mathcal{L}[\psi]-\mathcal{L}[\psi_{\mu}])\le 2(\mathcal{L}[\psi_{0}]-d_{\mathfrak{N}})<0, 
	\end{equation}
	 By (\ref{3.3}), (\ref{3.2}), and conservation laws, one deduces that 
	\begin{equation*}
		\mathcal{G}''(t)=8\mathcal{V}[\psi]\le 16(\mathcal{L}[\psi]-d_{\mathfrak{N}})<0. 
	\end{equation*}
	By  the definition of $\mathcal{G}(t)$, there is a finite time $T^{*}$ such that the corresponding  solution $\psi(t, x)$ to the Cauchy problem (\ref{Hartree equation})-(\ref{initial data}) blows up. 

	Next, we deal with $\psi_{0}\in\mathfrak{K^{+}}$, it follows from the Proposition \ref{invariant} that $\mathcal{L}[\psi]<d_{\mathfrak{N}}$, $\mathcal{N}[\psi]<0$, and $\mathcal{V}[\psi]>0$. Then, one deduces that 
	\begin{equation}
		\mathcal{L}[\psi]-\frac{1}{2}\mathcal{V}[\psi]=\frac{\omega}{2}\int |\psi|^{2}dx+\Big(\frac{D(p-1)}{4(p+1)}-\frac{1}{p+1}\Big)\int |\psi|^{p+1}dx<d_{\mathfrak{N}}, 
	\end{equation}
	then $\frac{\omega}{2}\int |\psi|^{2}dx<d_{\mathfrak{N}}$. 
	By utilizing the same scaling in Theorem \ref{Theorem 4.1}, one can see that $\lim\limits_{\lambda\to0^{+}}\mathcal{V}[\psi_{\lambda}]=-\infty$, then there exists $\lambda^{*}\in(0, 1)$ such that $\mathcal{V}[\psi_{\lambda^{*}}]=0$. $\mathcal{N}[\psi_{\lambda^{*}}]$ has following possibilities. 
	
	(i) $\mathcal{N}[\psi_{\lambda^{*}}]<0$. One can see that $\psi_{\lambda^{*}}\in\mathfrak{M}$, then $\mathcal{L}[\psi_{\lambda^{*}}]\ge d_{\mathfrak{M}}\ge d_{\mathfrak{N}}>\mathcal{L}[\psi]$, and for all time, one admits that 
	\begin{equation}
		\begin{split}
		\mathcal{L}[\psi]-\mathcal{L}[\psi_{\lambda^{*}}]&=\frac{1}{2}(1-\lambda^{*\frac{2D}{p+1}+2-D})\int |\nabla\psi|^{2}dx+\frac{\omega}{2}(1-\lambda^{*\frac{2D}{p+1}-D})\int |\psi|^{2}dx\\
		\vspace{0.2cm}
		&\quad-\frac{1}{4}(1-\lambda^{*\frac{4D}{p+1}+2-2D})\int (|x|^{-2}*|\psi|^{2})|\psi|^{2}dx\\
			\vspace{0.2cm}
		&<0.
		\end{split}
	\end{equation}
	By combining with $\frac{\omega}{2}\int |\psi|^{2}dx<d_{\mathfrak{N}}$, one derives $\int|\nabla\psi|^{2}dx\le C$. 
	
	(ii) $\mathcal{N}[\psi_{\lambda^{*}}]\ge 0$. One has 
	\begin{equation}
		\begin{split}
		\mathcal{L}[\psi_{\lambda^{*}}]-\frac{1}{2}\mathcal{V}[\psi_{\lambda^{*}}]&=\frac{\omega}{2}\lambda^{*\frac{2D}{p+1}-D}\int |\psi|^{2}dx+\Big(\frac{D(p-1)}{4(p+1)}-\frac{1}{p+1}\Big)\int |\psi|^{p+1}dx\\
			\vspace{0.2cm}
		&\le \lambda^{*\frac{2D}{p+1}-D}d_{\mathfrak{N}}. 
		\end{split}
	\end{equation}
	For $1+\frac{4}{D}\le p<3$, one dervies
	\begin{equation}
		\begin{split}
			\mathcal{L}[\psi_{\lambda^{*}}]-\frac{1}{p+1}\mathcal{N}[\psi_{\lambda^{*}}]&=(\frac{1}{2}-\frac{1}{p+1})\mathcal{K}[\psi_{\lambda^{*}}]+(\frac{1}{p+1}-\frac{1}{4})\int (|x|^{-2}*|\psi_{\lambda^{*}}|^{2})|\psi_{\lambda^{*}}|^{2}dx\\
				\vspace{0.2cm}
			&\le \lambda^{\frac{2D}{p+1}-D}d_{\mathfrak{N}}. 
		\end{split}
	\end{equation}
	For $3\le p<1+\frac{4}{D-2}$, one obtains
	\begin{equation}
		\begin{split}
			\mathcal{L}[\psi_{\lambda^{*}}]-\frac{1}{p+1}\mathcal{N}[\psi_{\lambda^{*}}]&=(\frac{1}{2}-\frac{1}{p+1})\mathcal{K}[\psi_{\lambda^{*}}]+(\frac{1}{4}-\frac{1}{p+1})\int |\psi_{\lambda^{*}}|^{p+1}dx\\
				\vspace{0.2cm}
			&\le \lambda^{\frac{2D}{p+1}-D}d_{\mathfrak{N}}. 
		\end{split}
	\end{equation}
	Above results show that, for all time, $\psi$ is bounded in $H^{1}(\mathbb{R}^{D})$. 
	
	Then, we handle the case $\psi_{0}\in\mathfrak{R^{+}}$. It follows from the Proposition \ref{invariant} that $\mathcal{L}[\psi]<d_{\mathfrak{N}}$ and $\mathcal{N}[\psi]>0$. Then, for all time, one has following estimate
	\begin{equation}
		\int |\psi|^{p+1}dx+\int (|x|^{-2}*|\psi|^{2})|\psi|^{2}dx<\mathcal{K}[\psi]. 
	\end{equation}
When $1+\frac{4}{D}\le p<3$,  one sees that  
	\begin{equation}
		\frac{1}{2}\mathcal{K}[\psi]-\frac{1}{4}(\int |\psi|^{p+1}dx+\int (|x|^{-2}*|\psi|^{2})|\psi|^{2}dx)<d_{\mathfrak{N}}, 
	\end{equation}
When $3\le p<1+\frac{4}{D-2}$, it follows that  
	\begin{equation}
		\frac{1}{2}\mathcal{K}[\psi]-\frac{1}{p+1}(\int |\psi|^{p+1}dx+\int (|x|^{-2}*|\psi|^{2})|\psi|^{2}dx)<d_{\mathfrak{N}}, 
	\end{equation}
	Therefore, the above two cases imply that $\psi$ is bounded in $H^{1}(\mathbb{R}^{D})$. 
	
	Finally, combining with above results, one can see that the solution $\psi(t, x)$ of the Cauchy problem (\ref{Hartree equation}) and (\ref{initial data}) globally exists in $t\in[0, +\infty)$
	by local well-posedness theory. 
	
\end{proof}

\begin{thm}
	Let $D\ge 3$. If $1+\frac{4}{D}\le p<1+\frac{4}{D-2}$ and $\mathcal{L}[\psi_{0}]<d_{\mathfrak{N}}$, then $\psi_{0}\in\mathfrak{K}$ iff $\psi(t, x)$ blows up in a finite time. 
\end{thm}
\begin{proof}
	This theorem is a direct result by Theorem \ref{Theorem 4.1} and Remark \ref{1.16}. 
	
\end{proof}

\begin{thm}
	Let $D\ge 3$. If $\mathcal{K}[\psi_{0}]<2d_{\mathfrak{N}}$, then the solution $\psi(t, x)$ to the Cauchy problem (\ref{Hartree equation}) and (\ref{initial data}) exists globally for all time. 
\end{thm}
\begin{proof}
	One can see that $\mathcal{L}[\psi_{0}]<d_{\mathfrak{N}}$, if there holds $\mathcal{N}[\psi_{0}]>0$, then $\psi_{0}\in\mathfrak{R^{+}}$. It follows from Theorem \ref{Theorem 4.1} that the corresponding solution $\psi(t, x)$ to the Cauchy problem (\ref{Hartree equation}) and (\ref{initial data}) exists globally for all time. Assume by contradiction, if $\mathcal{N}[\psi_{0}]\le 0$, then one can find a $\lambda\in(0, 1)$ such that $\mathcal{N}[\lambda\psi_{0}]=0$. It follows that  $\lambda\psi_{0}\in\mathfrak{N}$, which implies $\mathcal{L}[\lambda\psi_{0}]\ge d_{\mathfrak{N}}$. 
	
	However, the following observation 
	\begin{equation}
		\mathcal{K}[\lambda\psi_{0}]=\lambda^{2}\mathcal{K}[\psi_{0}]<2d_{\mathfrak{N}}, 
	\end{equation}
 implies $\mathcal{L}[\lambda\psi_{0}]<d_{\mathfrak{N}}$.  This leads a contradiction.  
	
\end{proof}

\section{Stability Result for Standing Wave}\label{section5}
In this section, we will investigate the stability of standing waves for Eq. \eqref{Hartree equation}. By utilizing profile decomposition theory, we obtain the orbital stability of standing waves, and we also derive the strong instability of standing waves by blow-up. 

\subsection{Orbital Stability}
In this subsection, we apply the profile decomposition theory to finish the proof of Theorem~\ref{Theorem 5.1}, and this method has been successfully applied in studying dynamical properties for many physical models, see~\cite{HK2005,Zhu2016}. Now, we introduce two useful propositions. 
\begin{prop}\label{prop 5.3}
	Let $D\ge 3$. If $1<p<1+\frac{4}{D}$ and $0<m<\|\nabla W\|_{2}^{2}$, then there exists $u\in H^{1}(\mathbb{R}^{D})$ such that $d_{m}=E[u]$. 
\end{prop}
\begin{proof}
	To begin with, the variational problem  (\ref{2.5}) is well-defined according to the sharp GN inequality (\ref{GN}) and (\ref{GGN}). Indeed, it follows from the energy functional that one has following estimate. 
	\begin{equation}
		\begin{split}
			\mathscr{E}[u]&=\frac{1}{2}\int |\nabla u|^{2}dx-\frac{1}{4}\int (|x|^{-2}*|u|^{2})|u|^{2}dx-\frac{1}{p+1}\int |u|^{p+1}dx\\
			&\ge \frac{1}{2}\int |\nabla u|^{2}dx-\frac{1}{2\| W\|_{2}^{2}}\|u\|_{2}^{2}\|\nabla u\|_{2}^{2}-\frac{1}{2\|R\|_{2}^{p-1}}\|u\|_{2}^{2-\frac{D-2}{2}(p-1)}\|\nabla u\|_{2}^{\frac{N(p-1)}{2}}\\
			&=\frac{1}{2}\Big(1-\frac{\|u\|_{2}^{2}}{\|W\|_{2}^{2}}\Big)\|\nabla u\|_{2}^{2}-\frac{1}{2\|R\|_{2}^{p-1}}\|u\|_{2}^{2-\frac{D-2}{2}(p-1)}\|\nabla u\|_{2}^{\frac{N(p-1)}{2}}\\
			&\ge \frac{1}{2}\Big(1-\frac{\|u\|_{2}^{2}}{\| W\|_{2}^{2}}\Big)\|\nabla u\|_{2}^{2}-\varepsilon\|\nabla u\|_{2}^{2}-C(\varepsilon, R, N, p), 
		\end{split}
	\end{equation}
	where $C(\varepsilon, R, N, p)$ is a constant. By taking $\varepsilon=\frac{1}{4}\Big(1-\frac{\|u\|_{2}^{2}}{\|W\|_{2}^{2}}\Big)$, one has following 
	\begin{equation}\label{5.6}
		\frac{1}{4}\Big(1-\frac{\|u\|_{2}^{2}}{\| W\|_{2}^{2}}\Big)\|\nabla u\|_{2}^{2}-C(\varepsilon, R, N, p)\le \mathscr{E}[u]. 
	\end{equation}
	Next, we select a sequence $\{u_{n}\}_{n=1}^{+\infty}$ that minimizes the functional associated with problem~(\ref{2.5}), then 
	\begin{equation}
		\|u_{n}\|_{2}^{2}= m, \quad \mathscr{E}[u_{n}]\to d_{m}, \quad\text{as}\quad n\to+\infty. 
	\end{equation}
	Thus, $d_{m}+1$ can be regarded as a upper bound of $\mathscr{E}[u_{n}]$ for $n$ large enough, combine this with (\ref{5.6}), one has
	\begin{equation*}
		\frac{1}{4}\Big(1-\frac{\|u_{n}\|_{2}^{2}}{\| W\|_{2}^{2}}\Big)\|\nabla u_{n}\|_{2}^{2}\le d_{m}+1+C(\varepsilon, R, N, p), 
	\end{equation*}
	which implies the sequence $\{\nabla u_{n}\}_{n=1}^{+\infty}$ is bounded in $L^{2}(\mathbb{R}^{D})$. 
	
	Let $u$ be a fixed function in $H^{1}(\mathbb{R}^{D})$. It follows that $\|u_{\lambda}\|_{2}^{2}=\|u\|_{2}^{2}=m$ under the scaling $u_{\lambda}=\lambda^{\frac{D}{2}}u(\lambda x)$, and after some basic computations, one has 
	\begin{equation*}
	     \begin{split}
	     		\mathscr{E}[u_{\lambda}]
	     		&=\lambda^{2}\Big(\frac{1}{2}\int |\nabla u|^{2}dx-\frac{1}{4}\int (|x|^{-2}*|u|^{2})|u|^{2}dx\Big)-\frac{1}{p+1}\lambda^{\frac{D}{2}(p-1)}\int |u|^{p+1}dx. 
	     \end{split}
	\end{equation*}
    It follows from the sharp GN inequality (\ref{GN}) that, if $\|u\|_{2}^{2}<\|\nabla W\|_{2}^{2}$, then there exists $C_{1}>0$ such that
    \begin{equation*}
    	\frac{1}{2}\int |\nabla u|^{2}dx-\frac{1}{4}\int (|x|^{-2}*|u|^{2})|u|^{2}dx\ge C_{1}>0. 
    \end{equation*}
    Noting that $0<\frac{D}{2}(p-1)<2$ for $1<p<1+\frac{4}{D}$, one can select $\lambda$ suitably such that $\mathscr{E}[u_{\lambda}]<0$. Since $\|u_{\lambda}\|_{2}^{2}=\|u\|_{2}^{2}=m$, one can deduce $d_{m}<0$ from (\ref{2.5}). For $n$ large enough,  there exists $\delta>0$ such that 
    \begin{equation*}
    	\begin{split}
    		\frac{1}{4}\int (|x|^{-2}*|u_{n}|^{2})|u_{n}|^{2}dx+\frac{1}{p+1}\int |u_{n}|^{p+1}dx&=\frac{1}{2}\int |\nabla u_{n}|^{2}dx\\
    		&\ge -d_{m}+\delta. 
    	\end{split}
    \end{equation*}
    For suitable $n$, there exists $C_{0}>0$ such that
    \begin{equation}
    	\frac{1}{4}\int (|x|^{-2}*|u_{n}|^{2})|u_{n}|^{2}dx+\frac{1}{p+1}\int |u_{n}|^{p+1}dx\ge C_{0}. 
    \end{equation}
    Above arguments show that $\{u_{n}\}_{n=1}^{+\infty}$ is bounded in $H^{1}(\mathbb{R}^{D})$, and from Proposition \ref{prop 5.2},   $\{u_{n}\}_{n=1}^{+\infty}$ can be decomposed by
    	\begin{equation}\label{5.9}
    	u_{n}(x)=\sum_{j=1}^{J}U^{j}(x-x_{n}^{j})+u_{n}^{J}(x).
    \end{equation}
    Injecting into the energy functional, one derives
    \begin{equation*}
    	\mathscr{E}[u_{n}]=\sum_{j=1}^{J}\mathscr{E}[U^{j}]+\mathscr{E}[u_{n}^{J}]+o(1),
    \end{equation*}
    where $o(1)=o_{n}(1)\to 0$ as $n\to+\infty$. 
    
    Consider the scaling $U_{\lambda_{j}}^{j}=\lambda_{j}U^{j}$, where  $\lambda_{j}=\frac{m^{1/2}}{\|U^{j}\|_{2}}\ge 1$. One can find that
    \begin{equation*}
    	\|U_{\lambda_{j}}^{j}\|_{2}^{2}=m, \quad\lambda_{j}^{2}-1\ge \lambda_{j}^{p-1}-1,\quad \text{for}\quad 1<p<1+\frac{4}{D}. 
    \end{equation*}
    From the convergence of $\sum\limits_{j=1}^{J}\|U^{J}\|_{2}^{2}$, there exists $j_{0}\ge 1$ such that $\inf\limits_{j\ge 1}\lambda_{j}^{2}-1=\lambda_{j_{0}}^{2}-1$=$\Big(\frac{m^{1/2}}{\|U^{j_{0}}\|_{2}}\Big)^{2}-1$. 
    
    Injecting the scaling $U_{\lambda_{j}}^{j}=\lambda_{j}U^{j}$ into energy transfers it into 
    \begin{equation*}
    	\begin{split}
    		\mathscr{E}[U_{\lambda_{j}}^{j}]&=\frac{\lambda_{j}^{2}}{2}\int |\nabla U^{j}|^{2}dx-\frac{\lambda_{j}^{p+1}}{p+1}\int |U^{j}|^{p+1}dx-\frac{\lambda_{j}^{4}}{4}\int (|x|^{-2}*|U^{j}|^{2})|U^{j}|^{2}dx\\
    		&=\lambda_{j}^{2}\mathscr{E}[U^{j}]-\frac{\lambda_{j}^{2}(\lambda_{j}^{p-1}-1)}{p+1}\int |U^{j}|^{p+1}dx-\frac{1}{4}(\lambda_{j}^{4}-\lambda_{j}^{2})\int (|x|^{-2}*|U^{j}|^{2})|U^{j}|^{2}dx. 
    	\end{split}
    \end{equation*}
    Thus,  
    \begin{equation*}
    	\mathscr{E}[U^{j}]=\frac{\mathscr{E}[U_{\lambda_{j}}^{j}]}{\lambda_{j}^{2}}+\frac{\lambda_{j}^{p-1}-1}{p+1}\int |U^{j}|^{p+1}dx+\frac{1}{4}(\lambda_{j}^{2}-1)\int (|x|^{-2}*|U^{j}|^{2})|U^{j}|^{2}dx. 
    \end{equation*}
   Similarly, for the term $\mathscr{E}[v_{n}^{J}]$, it follows that 
    \begin{equation*}
    	\begin{split}
    		\mathscr{E}[u_{n}^{J}]&=\frac{\|u_{n}^{J}\|_{2}^{2}}{m}\mathscr{E}[\frac{m^{1/2}}{\|u_{n}^{J}\|_{2}}u_{n}^{J}]+\frac{\Big(\frac{m^{1/2}}{\|u_{n}^{J}\|_{2}}\Big)^{p-1}-1}{p+1}\int |U_{n}^{J}|^{p+1}dx\\
    		&\quad+\frac{1}{4}\Big(\frac{m}{\|u_{n}^{J}\|_{2}^{2}}-1\Big)\int (|x|^{-2}*|u_{n}^{J}|^{2})|u_{n}^{J}|^{2}dx\\
    		&\ge \frac{\|u_{n}^{J}\|_{2}^{2}}{m}\mathscr{E}[\frac{m^{1/2}}{\|u_{n}^{J}\|_{2}}u_{n}^{J}]+o(1).
    	\end{split}
    \end{equation*}
    Noting that the relation $\|U_{\lambda_{j}}^{j}\|_{2}^{2}=m=\|\frac{m^{1/2}}{\|u_{n}^{J}\|_{2}}u_{n}^{J}\|_{2}^{2}$, one has $\mathscr{E}[U_{\lambda_{j}}^{j}]\ge d_{m}$, $\mathscr{E}[\frac{m^{1/2}}{\|u_{n}^{J}\|_{2}}u_{n}^{J}]\ge d_{m}$, and $\liminf\limits_{n\to+\infty}\mathscr{E}[u_{n}]=d_{m}$. 
    Combine above estimates with $E[u_{n}]$.
    \begin{equation*}
    	\begin{split}
    		\mathscr{E}[u_{n}]&=\sum_{j=1}^{J}\mathscr{E}[V^{j}]+\mathscr{E}[u_{n}^{J}]+o(1)\\
    		&\ge \sum_{j=1}^{J}\frac{\mathscr{E}[U_{\lambda_{j}}^{j}]}{\lambda_{j}^{2}}+\frac{\lambda_{j}^{p-1}-1}{p+1}\int |U^{j}|^{p+1}dx+\frac{1}{4}(\lambda_{j}^{2}-1)\int (|x|^{-2}*|U^{j}|^{2})|U^{j}|^{2}dx\\
    		&\quad+\frac{\|u_{n}^{J}\|_{2}^{2}}{m}\mathscr{E}[\frac{m^{1/2}}{\|u_{n}^{J}\|_{2}}u_{n}^{J}]+o(1)\\
    		&\ge \sum_{j=1}^{J}\frac{d_{m}}{\lambda_{j}^{2}}+\inf_{j\ge 1}\frac{\lambda_{j}^{p-1}-1}{p+1}\sum_{j=1}^{J}\int |V^{j}|^{p+1}dx+\inf_{j\ge 1}\frac{1}{4}(\lambda_{j}^{2}-1)\sum_{j=1}^{J}\int (|x|^{-2}*|U^{j}|^{2})|U^{j}|^{2}dx\\
    		&\quad+\frac{\|u_{n}^{J}\|_{2}^{2}}{m}d_{m}+o(1)\\
    		&\ge \sum_{j=1}^{J}\frac{d_{m}}{m}\|U^{j}\|_{2}^{2}+\frac{\|u_{n}^{J}\|_{2}^{2}}{m}d_{m}+\inf_{j\ge 1}(\lambda_{j}^{2}-1)\int \frac{1}{p+1}|u_{n}|^{p+1}+\frac{1}{4}(|x|^{-2}*|u_{n}|^{2})|u_{n}|^{2}dx\\
    		&\quad-\frac{1}{p+1}\int |u_{n}^{J}|^{p+1}dx-\frac{1}{4}\int (|x|^{-2}*|u_{n}^{J}|^{2})|u_{n}^{J}|^{2}dx+o(1). 
    	\end{split}
    \end{equation*}
    Let $n\to+\infty$ and $J\to+\infty$, there exists $C_{0}>0$ such that 
    \begin{equation*}
    	d_{m}\ge d_{m}+C_{0}\Big(\frac{m}{\|U^{j_{0}}\|_{2}^{2}}-1\Big),
    \end{equation*}
which implies that there should hold $m\le \|U^{j_{0}}\|_{2}^{2}$. However, it follows from the decomposition (\ref{5.9}) that $m\ge \|U^{j_{0}}\|_{2}^{2}$, which implies $U^{j_{0}}\neq 0$ is the only term in the decomposition and $\|U^{j_{0}}\|_{2}^{2}=m$. Furthermore, one has $\mathscr{E}[U^{j_{0}}]=d_{m}$, which completes the proof. 

\end{proof}

\begin{prop}\label{prop 5.4}
	Let $D\ge 3$, $1<p<1+\frac{4}{D}$, and $0<m<\|\nabla W\|_{2}^{2}$, then $\forall\varepsilon>0$, there exists $\delta>0$ such that for any $\psi_{0}\in H^{1}(\mathbb{R}^{D})$, if 
	\begin{equation*}
		\inf_{u\in\Omega}\|\psi_{0}-u\|_{H^{1}(\mathbb{R}^{D})}<\delta, 
	\end{equation*}
	then 
	\begin{equation*}
		\inf_{u\in\Omega}\|\psi(t, x)-u\|_{H^{1}(\mathbb{R}^{D})}<\varepsilon, 
	\end{equation*}
	$\psi(t, x)$ is the corresponding solution of the Cauchy problem (\ref{Hartree equation}) with initial datum $\psi_{0}$. 
\end{prop}
\begin{proof}
	Above argument shows that if $\|\psi_{0}\|_{2}^{2}\le \| \nabla W\|_{2}^{2}$, then 
	\begin{equation*}
		\frac{1}{4}\Big(1-\frac{\|\psi_{0}\|_{2}^{2}}{\|\nabla W\|_{2}^{2}}\Big)\|\nabla\psi\|_{2}^{2}\le \mathscr{E}[\psi_{0}]+C(\varepsilon, R, N, p), 
	\end{equation*}
	which implies $\psi(t, x)$ exists globally for all time. 
	
	Now, we will finish this proof by contradiction. If there exists $\varepsilon_{0}>0$ and a initial sequence $\{\psi_{0, n}\}_{n=1}^{+\infty}$ such that 
	\begin{equation*}
		\inf_{u\in\Omega}\|\psi_{0, n}(x)-u(x)\|_{H^{1}(\mathbb{R}^{D})}<\frac{1}{n}
	\end{equation*}
	and 
	\begin{equation*}
		\inf_{u\in\Omega}\|\psi_{n}(t_{n}, x)-u\|_{H^{1}(\mathbb{R}^{D})}\ge \varepsilon_{0}. 
	\end{equation*}
	And it follows from the conservation laws that
	\begin{equation*}
		\int |\psi_{n}(t_{n}, x)|^{2}dx=\int |\psi_{0, n}(x)|^{2}dx\to \int |u|^{2}dx=m, 
	\end{equation*}
	\begin{equation*}
		\mathscr{E}[\psi_{n}(t_{n}, x)]=\mathscr{E}[\psi_{0, n}(x)]\to \mathscr{E}[u]=d_{m}. 
	\end{equation*}
	Consequently, the functional value associated with the sequence 
	 $\{\psi_{n}(t_{n}, x)\}_{n=1}^{+\infty}$ converges to the infimum of problem (\ref{2.5}). Thus, we can find a minimizer $u(x)\in\Omega$ satisfying 
	\begin{equation*}
		\inf_{u\in\Omega}\|\psi_{n}(t_{n}, x)-u\|_{H^{1}(\mathbb{R}^{D})}\to 0, 
	\end{equation*}
	as $n$ large enough. This leads a contradiction. 
	
\end{proof}

\begin{proof}\textit{(The proof of Theorem \ref{Theorem 5.1}) }Proposition \ref{prop 5.3} and Proposition \ref{prop 5.4} can derive the results directly. 

\end{proof}

\subsection{Strong Instability}\label{strong instability}
In this subsection, we derive that the strong instability of standing waves for all $\omega>0$. This proof is benefited from the results in Propositions \ref{noempty} and \ref{compare}, and this kind of open problem proposed in \cite{Zhang2005} is solved. 
\begin{proof}\textit{(The proof of Theorem \ref{Theorem 5.4})}
	To start with, one can find following relation by simple computation. 
	\begin{equation}\label{4.4}
		\frac{d}{d\lambda}\mathcal{L}[\lambda\psi]=\lambda^{-1}\mathcal{N}[\lambda\psi],  
	\end{equation}
	which implies the polynomial $\mathcal{L}[\lambda\psi]$ exists unique maximum $\mathcal{L}[\lambda^{*}\psi]$, and $\mathcal{N}[\lambda^{*}\psi]=0$ there. Then, it follows from \eqref{variational problem} that $\mathcal{N}[u]=0$, for any $\lambda>1$, one has $\mathcal{N}[\lambda u]<0$. Besides, $\mathcal{N}[u]=0$ and \eqref{4.4} implies that $\mathcal{L}[\lambda u]<\mathcal{L}[u]$ for any $\lambda>0$ and $\lambda\neq1$. Thus, one deduces that $\lambda u\in\mathfrak{K}$ from $\mathcal{N}[\lambda u]<0$ and $\mathcal{L}[\lambda u]<d_{\mathfrak{N}}$. Take $\lambda$ close to $1$ such that
	\begin{equation*}
		\|\lambda u-u\|_{H^{1}(\mathbb{R}^{D})}=(\lambda-1)\|u\|_{H^{1}(\mathbb{R}^{D})}<\varepsilon. 
	\end{equation*}
	Denote $\psi_{0}=\lambda u$, it follows from Theorem \ref{Theorem 4.1} that the corresponding solution $\psi(t, x)$ of the Cauchy problem (\ref{Hartree equation})-(\ref{initial data}) blows up in a finite time. 
	
\end{proof}

\vspace{0.8cm}

{\scriptsize \begin{tabular}{l}
    \ \ \textsc{Guoyi Fu} \\
     \textsc{School of Mathematical Sciences, Sichuan Normal University}, \\ 
     \textsc{Chengdu, Sichuan 610066, China.} \\
    \ \ {\it Email address}: \texttt{guoyifu\_22@163.com} \\[1em] 
        
    \ \ \textsc{Shanshan Fu} \\
      \textsc{School of Mathematical Sciences, Sichuan Normal University}, \\ 
    \textsc{Chengdu, Sichuan 610066, China.} \\
      \ \ \textsc{and} \\
       \textsc{School of Science, Xihua University}, \\ 
      \textsc{Chengdu, Sichuan 610039, China.} \\
    \ \ {\it Email address}: \texttt{shanshanfu3@163.com} \\[1em]
    
        \ \ \textsc{Xiaoguang Li} \\
     \textsc{School of Mathematical Sciences, Sichuan Normal University}, \\ 
     \textsc{Chengdu, Sichuan 610066, China.} \\
    \ \ {\it Email address}: \texttt{lixgmath@163.com} \\[1em]
        
      \ \ \textsc{Jian Zhang}\\
      \textsc{School of Mathematical Sciences},
       \textsc{University of Electronic Science and Technology of China}, \\
       \textsc{Chengdu, 611731, China.} \\
    \ \ {\it Email address}: \texttt{zhangjian@uestc.edu.cn} \\[1em]
        
     \ \ \textsc{Shihui Zhu} \\
     \textsc{School of Mathematical Sciences, Sichuan Normal University}, \\ 
      \textsc{Chengdu, Sichuan 610066, China.} \\
     \ \ {\it  Email address}: \texttt{shihuizhumath@sicnu.edu.cn; shihuizhumath@163.com}
\end{tabular}}
\end{document}